\documentclass[final,leqno,onefignum,onetabnum]{siamltex1213}

\usepackage{amsmath}
\usepackage{amsfonts}
\usepackage{subfigure}
\usepackage{enumerate}
\usepackage{graphicx}
\usepackage{url}
\usepackage{multirow}
\usepackage{array}
\usepackage{mathtools}
\newcommand{\Htwo}{{\mathcal{H}_{2}}}

\title{Recycling BiCGSTAB with an Application to
Parametric Model Order Reduction
\thanks{This material is based upon work supported by the National Science Foundation (USA) under Grant Number DMS-1025327 and Grant Number DMS-1217156. Also supported by Council of Scientific and Industrial Research (India) Grant Number 25/(0220)/13/EMR-II.}
}

\author{Kapil Ahuja\thanks{Discipline of Computer Science
and Engineering, Indian Institute of Technology Indore, Indore, India
(\email{kahuja@iiti.ac.in}).} \and
Peter Benner\thanks{Max Planck Institute for Dynamics of
Complex Technical Systems, Magdeburg, Germany
(benner@mpi-magdeburg.mpg.de).} \and
Eric de Sturler\thanks{Department of Mathematics, Virginia Tech, Blacksburg, USA
(\email{sturler@vt.edu}).} \and
Lihong Feng\thanks{Max Planck Institute for Dynamics of
Complex Technical Systems, Magdeburg, Germany
(feng@mpi-magdeburg.mpg.de).}}

\begin{document}
\maketitle
\slugger{mms}{xxxx}{xx}{x}{x--x}%slugger should be set to mms, siap, sicomp, sicon, sidma, sima, simax, sinum, siopt, sisc, or sirev

\begin{abstract}
Krylov subspace recycling is a process for accelerating the convergence of sequences of linear systems. Based on this technique, the recycling BiCG algorithm has been developed recently. Here, we now generalize and extend this recycling theory to BiCGSTAB. Recycling BiCG focuses on efficiently solving sequences of dual linear systems, while the focus here is on efficiently solving sequences of single linear systems (assuming non-symmetric matrices for both recycling BiCG and recycling BiCGSTAB).

As compared with other methods for solving sequences of single linear systems with non-symmetric matrices (e.g., recycling variants of GMRES), BiCG based recycling algorithms, like recycling B\-i\-C\-G\-S\-T\-A\-B, have the advantage that they involve a short-term recurrence, and hence, do not suffer from storage issues and are also cheaper with respect to the orthogonalizations.

We modify the BiCGSTAB algorithm to use a recycle space, which is built from left and right approximate invariant subspaces. Using our algorithm for a parametric model order reduction example gives good results. We show about 40\% savings in the number of matrix-vector products and about 35\% savings in runtime.
\end{abstract}

\begin{keywords}
Krylov subspace recycling, deflation, BiCGSTAB, model reduction, rational Krylov.
\end{keywords}

\begin{AMS}
65F10, 65N22, 93A15, 93C05.
\end{AMS}

\pagestyle{myheadings}
\thispagestyle{plain}
\markboth{AHUJA, BENNER, DE STURLER, FENG}{RECYCLING BICGSTAB}

\section{Introduction}
We focus on efficiently solving sequences of linear systems of the following type:
\begin{equation}\label{eq:initial3}
%    \mathcal{A}^{(i)}\ {\chi}^{(i)}_{(j)} = {b}^{(i)}_{(j)},
    {A}^{(\iota)}\ x^{(\iota, \kappa)} = {b}^{(\iota, \kappa)},
\end{equation}
where ${A}^{(\iota)} \in \mathbb{R}^{n \times n}$ varies with $\iota$; ${b}^{(\iota, \kappa)} \in \mathbb{R}^n$ varies with both $\iota$ and $\kappa$; the matrices ${A}^{(\iota)}$ are large, sparse, and non-symmetric; and the change from one system to the next is small.

Krylov subspace methods are usually used for solving such large and sparse linear systems. For linear systems with non-symmetric matrices, GMRES~\cite{saad86gmres} is one of the first choices, but it is generally not optimal with respect to the runtime. BiCGSTAB~\cite{vorst92bicgstab}
% QMR~\cite{freund1991qmr}, and TFQMR~\cite{freund1993tfqmr} algorithms
is competitive with GMRES, and in many cases performs better than GMRES in time. Also, it does not suffer from storage issues, which is a problem in GMRES.
% (encountered often in the experiments in~\cite{feng2012mems}).

Krylov subspace recycling is a technique for efficient solution of sequences of linear systems. Here, while solving one system in the sequence, approximate invariant subspaces of the matrix are selected and used to accelerate the convergence of the next system in the sequence. Since the matrices in the sequence do not change much, this provides substantial reduction in both the number of matrix-vector products and time. See~\cite{parks06gcrodr} and~\cite{ahuja12rbicg} for more about Krylov subspace recycling.

Here, we have a sequence of linear systems with non-symmetric matrices, and hence, GCRO-DR~\cite{parks06gcrodr} and GCROT~\cite{parks06gcrodr} algorithms, which are recycling variants of GMRES, are more suited.
% However, GCRO-DR is most efficient for sequences where both the matrix and the right hand side change at every step in the sequence. Here the matrix changes only every few steps, and so simplified variants of GCRO-DR, that is, G-DRvar1 and G-DRvar2, have been developed~\cite{feng2012Mems}.
However, since there is no optimal method in time for solving linear systems with non-symmetric matrices, there is no optimal method in time for solving sequences of such linear systems. Like GMRES, its recycling variants may also suffer from storage issues. Hence, we develop a recycling variant of BiCGSTAB based on the work in~\cite{ahuja11phd, ahuja09ms}.
% Hence, we develop a recycling variant of BiCGSTAB.
%and its derived algorithms, i.e., B\-i\-C\-G\-S\-T\-A\-B\-2 and B\-i\-C\-G\-S\-T\-A\-B({\em l}).
% Developing recycling variants of these algorithms will also increase the body of solvers for sequences of linear systems with non-symmetric matrices.

We explore the usage of recycling B\-i\-C\-G\-S\-T\-A\-B for parametric model order reduction (PMOR)~\cite{baur2011PMOR,feng:proceedings} that requires solution of systems of the form (\ref{eq:initial3}). We show about 40\% reduction in the number of matrix-vector products when using recycling as compared with not using recycling in B\-i\-C\-G\-S\-T\-A\-B. In terms of time, this translates to about 35\% savings in runtime.

In related work in this area (specific to BiCGSTAB),~\cite{gutnket2014bicg} discusses a variant of recycling BiCGSTAB  (using the terminology of deflation and augmentation instead of Krylov subspace recycling). There are three main differences between that approach and ours. First, we use a different inner product in our derivation as compared with~\cite{gutnket2014bicg}. Second, the focus in~\cite{gutnket2014bicg} is on only using a recycle space, while here we discuss both using and generating a recycle space. Third, we also give numerical experiments demonstrating the usefulness of our approach, while~\cite{gutnket2014bicg} discusses only a theoretical framework. Also,~\cite{abdel2007lefteig, morgan2010nlandr, abdel2014rhs} focus on variants of deflated BiCGSTAB for multiple right hand sides and do not discuss changing matrices. 

Since we propose recycling BiCGSTAB as an alternative for GMRES-based recycling solvers, we also compare with GCRO-DR. In the context of PMOR, simplified versions of GCRO-DR have also been proposed~\cite{feng:proceedings, feng2012Mems}. We show that our recycling BiCGSTAB is 10\% more efficient in time than GCRO-DR for our test problem.

%%%%%%%%%%%%%%%%%%%%% KAPIL May 23rd 2014 NEED TO FILL THIS WHEN REVIWERS COMMENTS COME%%%%%%%%%%%%%%%%
%
% ??? Discussion of block methods for multiple RHS needed ???
%
%%%%%%%%%%%%%%%%%%%%% KAPIL May 23rd 2014 NEED TO FILL THIS WHEN REVIWERS COMMENTS COME%%%%%%%%%%%%%%%%

%The recycling BiCGSTAB developed here will be useful for other application areas as well where systems of type (\ref{eq:initial3}) arise (e.g., acoustics).

To simplify notation, we drop the superscripts $\iota$ and $\kappa$ in (\ref{eq:initial3}).
% At any particular point in the sequence of systems, we refer to $A x = b$ as the primary system and $A^*\tilde{x} = \tilde{b}$ as the dual system.
Throughout the paper, $||\cdot||$ refers to the two-norm, $(\cdot, \cdot)$ refers to the standard inner product, $*$ indicates the conjugate transpose operation, $\bar{\cdot}$ indicates complex conjugation, and 
$\underline{\cdot}$ is used to signify a rectangular matrix. 
% for a matrix $\underline{H}_i$ the index $i$ indicates the number of columns and the underline indicates an extra row (i.e., $i+1$ rows). 
%, and $e_i$ is the $i$-th canonical basis vector. Unless otherwise stated, we collectively call the primary system recycle space and the dual system recycle space as the recycle space.

The rest of the paper is divided into six more sections. The bi-Lanczos algorithm~\cite{lanczos1952base} and recycling BiCG~\cite{ahuja12rbicg} form the basis of our recycling BiCGSTAB. Hence, we revisit these in Sections \ref{sec:bilanczos} and \ref{sec:rbicg}, respectively. In Section \ref{sec:rbicg}, we also give a new result related to recycling BiCG. Next, we derive recycling BiCGSTAB in Section \ref{sec:rbicgstab}. In Section \ref{sec:analysis}, we analyze the subspaces that can be used in recycling BiCGSTAB. Finally, we discuss the application of recycling BiCGSTAB to PMOR in Section \ref{sec:appl}, and give concluding remarks in Section \ref{sec:conclusion}.

\section{The Bi-Lanczos Algorithm}\label{sec:bilanczos}
Consider a primary system $Ax = b$, with $x_0$ the initial guess and $r_0 = b - A x_0$ the residual. Also, consider an auxiliary dual system $A^*\tilde{x} = \tilde{b}$, with $\tilde{b}$ a random vector, $\tilde{x}_0$ the initial guess, and $\tilde{r}_0 = \tilde{b} - A^* \tilde{x}_0$ the residual. This dual system is termed auxiliary because for this work we are not interested in its solution (although the system is real). The bi-Lanczos algorithm remains the same even when the dual system is of interest.

% \footnote{The system is real, but we are not interested in its solution.}

Let the columns of $V_{i} = [v_{1}\ v_{2}\ \ldots\ v_i]$ define the basis of the primary system Krylov space $\mathcal{K}^{i}(A,r_{0}) \equiv span\{r_{0},\ Ar_{0},\ A^2 r_{0},\ \cdots,\ A^{i-1} r_{0}\}$. Also, let the columns of ${\tilde V}_{i} = [{\tilde v}_{1}\ {\tilde v}_{2}\ \ldots\ \tilde{v}_i]$ define the basis of the dual system Krylov space $\mathcal{\tilde{K}}^{i}(A^*,\tilde{r}_{0}) \equiv span\{\tilde{r}_{0},\ A^*\tilde{r}_{0},\ A^{2*} \tilde{r}_{0},\ \cdots,\ A^{(i-1)*} \tilde{r}_{0}\}$.

The bi-Lanczos algorithm computes the columns of $V_i$ and $\tilde{V}_i$ such that, in exact arithmetic, $V_{i} \perp_{b} \tilde{V}_{i}$, where $\perp_{b}$ is referred to as bi-orthogonality; this implies that $\tilde{V}_i^*V_i$ is a diagonal matrix. The columns of $V_i$ and $\tilde{V}_i$   are called Lanczos vectors. There is a degree of freedom in choosing the scaling of the Lanczos vectors~\cite{greenbaum1997book, gutnket1997survey, saad03book}. Using the scaling
\begin{equation}\label{scalingfactor}
    ||v_{i}|| = 1, \quad \quad (v_i, \tilde{v}_i) = 1,
\end{equation}
we initialize the Lanczos vectors as follows:
\begin{align*}
    \begin{array}
        [l]{ll}
        v_1 = \frac{r_0}{|| r_0 ||}, & \tilde{v}_{1} = \frac{\tilde{r}_{0}}{(v_{1}, \tilde{r}_{0})}.
    \end{array}
\end{align*}
The ($i+1$)-at Lanczos vectors are given by
\begin{align}\label{eq:biLanczosCompute}
    \begin{array}
        [l]{lll}
        \gamma v_{i+1} & = & A v_i - V_i \tau \perp \tilde{V}_i, \\
        \tilde{\gamma}\tilde{v}_{i+1} & = & A^* \tilde{v}_i - \tilde{V_i}\tilde{\tau}  \perp {V}_i,
    \end{array}
\end{align}
where $\gamma$, $\tilde{\gamma}$ and $\tau$, $\tilde{\tau}$ are are determined by the biorthogonality condition (\ref{eq:biLanczosCompute}) and the normalization condition (\ref{scalingfactor}). The computation of the ($i+1$)-st Lanczos vectors requires only the $i$-th and the ($i-1$)-st Lanczos vectors (see~\cite{saad03book}). These $3$-term recurrences are called the bi-Lanczos relations, and are defined as follows:
\begin{align}\label{eq:biLanczos}
    \begin{array}
        [l]{lllll}
            AV_{i} & = & V_{i+1}\underline{T}_{i} & =  & V_{i}{T}_{i} + t_{i+1,i}v_{i+1}e_i^T, \\
            A^*\tilde{V}_{i} & = & \tilde{V}_{i+1}\tilde{\underline{T}}_{i} & = & \tilde{V}_{i}{\tilde{T}}_{i} + \tilde{t}_{i+1,i}\tilde{v}_{i+1}e_i^T, \\
    \end{array}
\end{align}
where $T_i$, $\tilde{T}_i$ are $i \times i$ tridiagonal matrices, $t_{i+1,i}$ is the last element of the last row of $\underline{T}_i \in \mathbb{C}^{(i+1) \times i}$, and $\tilde{t}_{i+1,i}$ is the last element of the last row of $\tilde{\underline{T}}_i \in \mathbb{C}^{(i+1) \times i}$.

The bi-Lanczos algorithm breaks down when at any step $i$, $\tilde{v}_i^*v_i$ = 0. There exist so-called look-ahead strategies~\cite{freund:lookahead, gutnket1997survey}, that take multiple Lanczos vectors together in-succession and make them block bi-orthogonal, to avoid this breakdown. 

%%%%%%%%%%%%%%%%%%%%%%%%%%%%%%%%%%%%%%%%%%%%%%%%

\section{Recycling BiCG Revisited}\label{sec:rbicg}
We first introduce a generalization of the bi-Lanczos algorithm~\cite{ahuja11phd}. We show that even for a pair of matrices that are not conjugate transposes of each other,
% satisfying some mild conditions,
one can build bi-orthogonal bases (for the associated two Krylov subspaces) using a short-term recurrence.

Expanding the search space to include a recycle space leads to an augmented bi-orthogonality condition. The augmented bi-Lanczos algorithm, as derived for recycling BiCG~\cite{ahuja12rbicg}, computes bi-orthogonal bases for the two Krylov subspaces such that this augmented bi-orthogonality condition is satisfied. Next, we revisit augmented bi-Lanczos~\cite{ahuja09ms} and show that it is a special case of generalized bi-Lanczos.  Finally, we list the recycling BiCG algorithm from~\cite{ahuja12rbicg}.

% Subsection:The Generalized Bi-Lanczos
There are numerous ways of computing good bases for Krylov subspaces $\mathcal{K}^{m}(B,v_{1})$ and $\mathcal{K}^{m}(\tilde{B},\tilde{v}_{1})$, where $B$ and $\tilde{B}$ are $n \times n$ general matrices, and $v_1$ and $\tilde{v}_1$ are any two $n$ dimensional vectors. Let the columns of $V_{m} = [v_{1}\ v_{2}\ \ldots\ v_m]$ and $\tilde{V}_{m} = [\tilde{v}_{1}\ \tilde{v}_{2}\ \ldots\ \tilde{v}_m]$ define one such pair of good bases for $\mathcal{K}^{m}(B,v_{1})$ and $\mathcal{K}^{m}(\tilde{B},\tilde{v}_{1})$, respectively. We compute these bases using the following, in principle, full recurrences:
\begin{align}\label{eq:basis1}
    \begin{array}
        [l]{lll}
	\beta_{i+1,i}{v}_{i+1} & = & {B}{v}_i - {\beta}_{ii}{v}_i -  {\beta}_{i-1,i}{v}_{i-1} - \ldots - {\beta}_{1i}{v}_1,
	\end{array}
\end{align}
\begin{align}\label{eq:basis2}
    \begin{array}
        [l]{lll}
	\tilde{\beta}_{i+1,i}\tilde{v}_{i+1} & = & \tilde{B}\tilde{v}_i - \tilde{\beta}_{ii}\tilde{v}_i -  \tilde{\beta}_{i-1,i}\tilde{v}_{i-1} - \ldots - \tilde{\beta}_{1i}\tilde{v}_1,
	\end{array}
\end{align}
where $i \in \{1, 2, 3, \ldots,m-1\}$ and $\{\beta_{ij}\}$, $\{\tilde{\beta}_{ij}\}$ are scalars to be determined. We assume that for $i < m$, $\mathcal{K}^{i}(B,v_{1})$ is not an invariant subspace of $B$ (similarly, $\mathcal{K}^{i}(\tilde{B},\tilde{v}_{1})$ is not an invariant subspace of $\tilde{B}$ for $i < m$). We can rewrite (\ref{eq:basis1}) as follows:
\begin{equation*}
	{B}{v}_i  = {\beta}_{1i}{v}_1 + {\beta}_{2i}{v}_2 + \ldots + {\beta}_{i-1,i}{v}_{i-1} + {\beta}_{ii}{v}_i  +  {\beta}_{i+1,i}{v}_{i+1}.
\end{equation*}
Combining these equations, for $i \in \{1, 2, 3, \ldots, m-1\}$, into matrix form we get
\begin{align*}
   \begin{array}
	[l]{l}
	B[v_1 \ v_{2} \ldots v_{m-1}] =
	% \\ \quad \\
   	\left[v_1 \ v_{2} \ldots v_{m-2} \ v_{m-1} \ v_{m}\right]
	\left[
       \begin{array}
        	[c]{ccccc}%
        	\beta_{11} & \beta_{12}  & \ldots      & \beta_{1,m-1} \\
        	\beta_{21} & \beta_{22}  & \ldots      & \beta_{2,m-1} \\
        	0         	   & \beta_{32}   & \ldots     & \beta_{3,m-1} \\
        	0         	   & 0                 & \ldots     & \beta_{4,m-1} \\
        	\vdots       & \ddots          & \ddots    & \vdots \\
        	0               &  0                &  \ddots 	 & \beta_{m-1,m-1} \\
        	0               &  0                &  \ldots     & \beta_{m,m-1} \\
      \end{array}
      \right],
   \end{array}
\end{align*}
or
$$BV_{m-1} = V_{m}\underline{H}_{m-1},$$
where $\underline{H}_{m-1}$ is an $m \times (m-1)$ upper Hessenberg matrix. This result also holds for each $i \in \{1, 2, 3, \ldots, m-1\}$, i.e.,
\begin{align}\label{eq:generalRelations1}
    \begin{array}
       [l]{lll}
	BV_i & = & V_{i+1}\underline{H}_i.
    \end{array}
\end{align}
Similarly, using (\ref{eq:basis2}) and following the steps above, we get the following relation for the dual system:
\begin{align}\label{eq:generalRelations2}
    \begin{array}
       [l]{lll}
	\tilde{B}\tilde{V}_i & = & \tilde{V}_{i+1}\tilde{\underline{H}}_i.
    \end{array}
\end{align}
The scalars $\{\beta_{ij}\}$ and $\{\tilde{\beta}_{ij}\}$ are determined by a choice of constraints. One option is to enforce that the columns of $V_i$ (and $\tilde{V}_i$) are orthonormal vectors (as in the Arnoldi algorithm). Another option, as in the bi-Lanczos algorithm, is to enforce\footnote{In this paper, for ease of exposition, we assume breakdowns do not happen. Hence, $(\tilde{v}_i, v_i) \neq  0$.}
% we can compute scalars in (\ref{eq:generalRelations}) by using (\ref{eq:findScalars}). for all $i$
\begin{equation*}
	V_i \perp_b \tilde{V}_i, \quad ||v_i|| = 1, \quad \text{and} \quad  (v_i, \tilde{v}_i) = 1,
\end{equation*}
or
\begin{equation}\label{eq:findScalars}
	\tilde{V}_{i}^*V_{i} = I \quad \text{and} \quad  ||v_i|| = 1.
\end{equation}
If $\tilde{B} = B^*$, then (\ref{eq:generalRelations1}), (\ref{eq:generalRelations2}), and (\ref{eq:findScalars}) lead to the bi-Lanczos relations (\ref{eq:biLanczos}), which consist of three-term recurrences. Our goal here is to relax the condition $\tilde{B} = B^*$ and still obtain short-term recurrences.

\begin{theorem}\label{thm:genBi}
	Let $B, \tilde{B} \in \mathbb{C}^{n \times n}$, and let the following conditions
	% (a) -- (c)
	hold:
	\begin{enumerate}[(a)]
		\item $B  - \tilde{B}^* = \tilde{F}_k \tilde{C}_k^* - C_k F_k^*$, where $C_k, \tilde{C}_k, F_k, \tilde{F}_k  \in \mathbb{C}^{n \times k}$,
		\item $\forall{x}: B x \perp \tilde{C}_k$, $\forall\tilde{x}: \tilde{B} \tilde{x} \perp C_k$,
	 	\item $v_1 \perp \tilde{C}_k$, and $\tilde{v}_1 \perp C_k$.
	\end{enumerate}
	Also, let (\ref{eq:findScalars}) be used as the set of constraints for (\ref{eq:generalRelations1}) and (\ref{eq:generalRelations2}). Then, $\beta_{ij} = 0$ and $\tilde{\beta}_{ij} = 0$  for $j > {i+1}$, which leads to the following three-term recurrences:
	\begin{align*}
	    \begin{array}
	    [l]{lll}
	    \beta_{i+1,i}{v}_{i+1} & = & {B}{v}_i - {\beta}_{ii}{v}_i -  {\beta}_{i-1,i}{v}_{i-1}, \\
	    \tilde{\beta}_{i+1,i}\tilde{v}_{i+1} & = & \tilde{B}\tilde{v}_i - \tilde{\beta}_{ii}\tilde{v}_i -  \tilde{\beta}_{i-1,i}\tilde{v}_{i-1},
	    \end{array}
	\end{align*}
	for $i \in \{1, 2, 3, \ldots, m-1\}$.
\end{theorem}
\begin{proof}
	Using (b) and (c) we can show that % "We will show that" is used in Dr. Gugercin's paper page 624.
	\begin{equation}\label{eq:CVBiOrth}
		C_k^* \tilde{V}_i = 0 \quad \text{and} \quad \tilde{C}_k^* V_i = 0.
	\end{equation}
	We show $C_k^* \tilde{V}_i =0$ by induction. One can similarly show that $\tilde{C}_k^* V_i = 0$. $C_k^* \tilde{v}_1 = 0$ by (c). Let $C_k^*\tilde{v}_l = 0$ for $l = \{1, 2,\ldots,i\}$, and consider the case $l = i+1$. From (\ref{eq:basis2}) we know that
	\begin{equation*}
		\tilde{\beta}_{i+1,i}\tilde{v}_{i+1} = \tilde{B}\tilde{v}_i - \tilde{\beta}_{ii}\tilde{v}_i -  \tilde{\beta}_{i-1,i}\tilde{v}_{i-1} - \ldots - \tilde{\beta}_{1i}\tilde{v}_1.
	\end{equation*}
	Then, $C_k^* \tilde{v}_{i+1}  = 0$ since $C_k^* \tilde{B}\tilde{v}_i  = 0$ using (b) and $\tilde{\beta}_{li}C_k^*\tilde{v}_l = 0$ for $l \in \{1, 2,\ldots,i\}$ by the induction hypothesis\footnote{Note that our earlier assumption, $\mathcal{K}^{i}(\tilde{B},\tilde{v}_{1})$ is not an invariant subspace of $\tilde{B}$ for $i < m$, shows that $\tilde{\beta}_{i+1,i} \ne 0$.}. This proves (\ref{eq:CVBiOrth}). Multiplying both sides in (\ref{eq:generalRelations1}) by $\tilde{V}_i^*$ and using (\ref{eq:findScalars}) we get
	\begin{equation*}
		\tilde{V}_i^*BV_i = H_i.
	\end{equation*}
	Substituting (a) in the above equation leads to
	\begin{align*}
		\begin{array}
		[l]{llll}
		\tilde{V}_i^*\left(\tilde{B}^* + \tilde{F}_k \tilde{C}_k^* - C_k F_k^*\right)V_i & = & H_i & \Longleftrightarrow \\
		\tilde{V}_i^*\tilde{B}^*V_i + \tilde{V}_i^*\tilde{F}_k \tilde{C}_k^*V_i - \tilde{V}_i^*C_k F_k^*V_i & = & H_i. &
		\end{array}
	\end{align*}
	Using (\ref{eq:CVBiOrth}) we get
	\begin{align*}
		\begin{array}
		[l]{llll}
		\tilde{V}_i^*\tilde{B}^*V_i & = & H_i & \Longleftrightarrow \\
		(\tilde{B}\tilde{V}_i)^*V_i & = & H_i. &
		\end{array}
	\end{align*}
	Finally, using (\ref{eq:generalRelations2}) and (\ref{eq:findScalars}) in the above equation gives
	\begin{equation*}
		\tilde{H}_i^* = H_i.
	\end{equation*}
	This implies both $H_i$ and $\tilde{H}_i$ are tridiagonal matrices, and hence $\beta_{ij} = 0$ and $\tilde{\beta}_{ij} = 0$  for $j > {i+1}$.
\end{proof}

% Subsection:The Augmented Bi-Lanczos
We now revisit augmented bi-Lanczos~\cite{ahuja09ms} and show that it is a special case of generalized bi-Lanczos. The BiCG algorithm is primarily used where the dual system is not auxiliary. That is, one needs to solve both a primary system and a dual system. The recycling BiCG algorithm (also termed RBiCG) was developed to accelerate the convergence of sequences of such systems.

In RBiCG, we use the matrix $U$ to define the primary system recycle space, and compute $C = A^{(\iota+1)}U$, where $U$ is derived from an approximate right invariant subspace of $A^{(\iota)}$ and $\iota$ denotes the index of the linear system in the sequence of linear systems; see (\ref{eq:initial3}). Similarly, we use the matrix $\tilde{U}$ to define the dual system recycle space, and compute $\tilde{C} = A^{(\iota+1)*} \tilde{U}$, where $\tilde{U}$ is derived from an approximate left invariant subspace of $A^{(\iota)}$. $U$ and $\tilde{U}$ are computed such that $C$ and $\tilde{C}$ are bi-orthogonal (see page 35 of~\cite{ahuja11phd}). The number of vectors selected for recycling is denoted by $k$, and hence, $U$, $\tilde{U}$, $C$, and $\tilde{C} \in \mathbb{C}^{n \times k}$.

The bi-Lanczos algorithm was modified to compute the columns of $V_i$ and $\tilde{V}_i$ such that
\begin{equation*}\label{biorth}
    \left[C\ V_{i}\right] \perp_{b} \left[\tilde{C}\ \tilde{V}_{i}\right].
\end{equation*}
% $C \perp_b \tilde{C}$ is easy to implement when computing the recycle space.
Using the scaling (\ref{scalingfactor}), we initialize the Lanczos vectors as
\begin{align*}
    \begin{array}
        [c]{cc}
        v_1 = \frac{\textstyle \left(I - C \mathcal{D}_c^{-1} \tilde{C}^*\right)r_0}{\textstyle \left|\left|\left(I - C \mathcal{D}_c^{-1} \tilde{C}^*\right)r_0\right|\right|}, & \tilde{v}_{1} = \frac{\displaystyle \left(I - \tilde{C} \mathcal{D}_c^{-1} C^*\right)\tilde{r_0}}{\displaystyle \left(v_{1}, \left(I - \tilde{C} \mathcal{D}_c^{-1} C^*\right)\tilde{r_0}\right)}.
    \end{array}
\end{align*}
Here $\mathcal{D}_c = \tilde{C}^* C$ is a diagonal matrix (implied by $C \perp_b \tilde{C}$; we also enforce $\mathcal{D}_c$ to have positive, real coefficients). As for the bi-Lanczos algorithm in (\ref{eq:biLanczosCompute}), the ($i+1$)-st Lanczos vectors here are given by
\begin{align*}
    \begin{array}
        [l]{lll}
	\gamma {v}_{i+1} & = & A v_i - V_i \tau - C \rho \perp \left[\tilde{C}\ \tilde{V}_i\right], \\
	\tilde{\gamma} \tilde{v}_{i+1} & = & A^* \tilde{v}_i - \tilde{V}_i \tilde{\tau} - \tilde{C} \tilde{\rho} \perp \left[{C}\ {V}_i\right], \\
    \end{array}
\end{align*}
where $\gamma$, $\tilde{\gamma}$, $\tau$, $\tilde{\tau}$, $\rho$, and $\tilde{\rho}$ are to be determined. The computation of the ($i+1$)-st Lanczos vector for the primary system now requires the $i$-th and ($i-1$)-st Lanczos vectors and $C$ (see~\cite{ahuja09ms}). This gives a ($3+k$)-term recurrence, where $k$ is the number of columns of $C$. Similarly, we get a ($3+k$)-term recurrence for computing the Lanczos vectors for the dual system. We refer to this pair of ($3+k$)-term recurrences as the augmented bi-Lanczos relations, and they are given by
\begin{align*}
	\begin{array}[l]{lll}
	(I-C \hat{C}^*)AV_i & = & V_{i+1}\underline{T}_i, \\[5pt]
	(I-\tilde{C}\check{C}^*)A^* \tilde{V}_i & = & \tilde{V}_{i+1}\tilde{\underline{T}}_i,
	\end{array}
\end{align*}
where
\begin{align}\label{eq:CcheckChat}
	\begin{array}[l]{lll}
	\hat{C}=\left[\frac{1}{c_{1}^*\tilde{c}_{1}}\tilde{c}_{1}\ \frac{1}{c_{2}^*\tilde{c}_{2}}\tilde{c}_{2}\ \cdots\ \frac{1}{c_{k}^*\tilde{c}_{k}}\tilde {c}_{k}\right] & = & \tilde{C}\mathcal{D}_c^{-1}, \\[5pt]
	\check{C}=\left[\frac{1}{\tilde{c}_{1}^*c_{1}}c_{1}\ \frac{1}{\tilde{c}_{2}^*c_{2}}c_{2}\ \cdots\ \frac{1}{\tilde{c}_{k}^*c_{k}}c_{k}\right] & = & {C}\mathcal{D}_c^{-1}.
	\end{array}
\end{align}

\begin{theorem}\label{thm:genToAug}
 	Let $v_{1} = \eta (I - C \mathcal{D}_c^{-1} \tilde{C}^*)r_0$, $\tilde{v}_1 =  \tilde{\eta}(I - \tilde{C} \mathcal{D}_c^{-1} C^*)\tilde{r}_0$, $B = (I - C \mathcal{D}_c^{-1} \tilde{C}^*)A$, and $\tilde{B} = (I - \tilde{C} \mathcal{D}_c^{-1} C^*)A^*$, where $\eta$, $\tilde{\eta}$ are scalars and $C$, $\tilde{C} \in \mathbb{C}^{n \times k}$ s.t. $\mathcal{D}_c = \tilde{C}^* C$ is a diagonal matrix with positive, real coefficients. 	Also, let (\ref{eq:findScalars}) be used as the set of constraints for (\ref{eq:generalRelations1}) and (\ref{eq:generalRelations2}). Then, $\beta_{ij} = 0$ and $\tilde{\beta}_{ij} = 0$  for $j > {i+1}$, which leads to the following short-term recurrences:
	\begin{align*}
	    \begin{array}
	    [l]{lll}
	    \beta_{i+1,i}{v}_{i+1} & = & {B}{v}_i - {\beta}_{ii}{v}_i -  {\beta}_{i-1,i}{v}_{i-1}, \\
	    \tilde{\beta}_{i+1,i}\tilde{v}_{i+1} & = & \tilde{B}\tilde{v}_i - \tilde{\beta}_{ii}\tilde{v}_i -  \tilde{\beta}_{i-1,i}\tilde{v}_{i-1},
	    \end{array}
	\end{align*}
	for $i \in \{1, 2, 3, \ldots, m-1\}$.
\end{theorem}
\begin{proof}
	We show that conditions (a) -- (c) of Theorem \ref{thm:genBi} are satisfied. This demonstrates that augmented bi-Lanczos is a special case of generalized bi-Lanczos. We have $B, \tilde{B} \in \mathbb{C}^{n \times n}$ such that
	\begin{align*}
		\begin{array}
		[l]{lll}
		B - \tilde{B}^* & = & A -  C \mathcal{D}_c^{-1} \tilde{C}^*A - A + A C\mathcal{D}_c^{-1}\tilde{C}^* \\
					& = & \left(A C\mathcal{D}_c^{-1}\right)\tilde{C}^* -  C \left(A^*\tilde{C}\mathcal{D}_c^{-1}\right)^*.
		\end{array}
	\end{align*}
	Defining $F = A^*\tilde{C}\mathcal{D}_c^{-1}$ and $\tilde{F} = A C\mathcal{D}_c^{-1}$, we get
	$$B - \tilde{B}^* = \tilde{F} \tilde{C}^* - C F^* \quad \text{where} \quad C, \tilde{C}, F, \tilde{F}  \in \mathbb{C}^{n \times k}.$$
	Hence (a) is satisfied. For any $\tilde{x}$ consider the following:
	\begin{align*}
		\begin{array}
		[l]{lll}
		C^*\tilde{B}\tilde{x} & = & C^*(I - \tilde{C} \mathcal{D}_c^{-1} C^*)A^*\tilde{x}\\
						  & = & (C^* - \mathcal{D}_c \mathcal{D}_c^{-1} C^*)A^*\tilde{x} = 0.
		\end{array}
	\end{align*}
	Similarly, for any $x$ consider the following:
	\begin{align*}
		\begin{array}
		[l]{lll}
		\tilde{C}^*Bx & = & \tilde{C}^*(I - C \mathcal{D}_c^{-1} \tilde{C}^*)Ax\\
					& = & (\tilde{C}^* - \mathcal{D}_c \mathcal{D}_c^{-1} \tilde{C}^*)Ax = 0.
		\end{array}
	\end{align*}
Hence (b) is satisfied. Similarly, for $v_1$ and $\tilde{v}_1$ chosen in the theorem, $\tilde{C}^*v_1 = 0$ and $C^*\tilde{v}_1 = 0$. Hence, (c) is satisfied.
 \end{proof}

 % Subsection:RBiCG
 For ease of future derivations, we introduce a slight change of notation. Let $x_{-1}$ and $\tilde{x}_{-1}$ be the initial guesses and $r_{-1} = b - A x_{-1}$ and $\tilde{r}_{-1} = \tilde{b} - A^* \tilde{x}_{-1}$ the corresponding initial residuals. We define
\begin{align}\label{minusXR}
    \begin{array}{lcr}
	    x_0 = x_{-1} + U \hat{C}^*r_{-1}, & \qquad & r_0 = (I - C\hat{C}^*)r_{-1}, \\
	    \tilde{x}_0 = \tilde{x}_{-1} + \tilde{U} \check{C}^*\tilde{r}_{-1}, & \qquad & \tilde{r}_0 = (I - \tilde{C}\check{C}^*)\tilde{r}_{-1},
    \end{array}
\end{align}
and follow this convention for $x_0$, $\tilde{x}_0$, $r_0$, and $\tilde{r}_0$ for the rest of the paper. Algorithm 1 gives the RBiCG algorithm from~\cite{ahuja12rbicg}. Here, we have not given details on how the recycle space is computed in RBiCG. For that we refer the reader to~\cite{ahuja12rbicg}.

Like BiCG, breakdowns can happen in RBiCG as well. Besides the breakdown in the underlying augmented bi-Lanczos algorithm ($\tilde{v}_i^*v_i$ = 0 at step $i$; called {\it serious breakdown}), a breakdown can happen when pivotless LDU decomposition of the  tridiagonal matrix (as in the augmented bi-Lanczos relations discussed earlier) does not exist. This is referred to as a breakdown of the {\it second kind}. 

The breakdown in the augmented bi-Lanczos algorithm can be avoided by using look-ahead strategies~\cite{freund:lookahead, gutnket1997survey} (as applied for the bi-Lanczos algorithm). The second breakdown can also be avoided in the same way as in BiCG. That is,  by performing the LDU decomposition with $2 \times 2$ block diagonal elements~\cite{bank1993csb}. 
 
\begin{figure}
{\bf Algorithm 1.} {\it RBiCG~\cite{ahuja12rbicg}} \\
    1. Given $U$ (also $C = AU$) and $\tilde{U}$ (also $\tilde{C} = A^* \tilde{U}$) s.t. $C \perp_b \tilde{C}$, compute $\check{C}$ and $\hat{C}$ using (\ref{eq:CcheckChat}).  If $U$ and $\tilde{U}$ are not available, then initialize $U$, $\tilde{U}$, $\check{C}$, and $\hat{C}$ to empty matrices. \\
    2. Choose $x_{-1}$, $\tilde{x}_{-1}$ and compute $x_0$, $\tilde{x}_0$, $r_0$, and $\tilde{r}_{0}$ using (\ref{minusXR}). \\
    3. {\bf if} $(r_0, \tilde{r}_0) = 0$ {\bf then} initialize $\tilde{x}_{-1}$ to a random vector. \\
    4. Set $p_0 = 0$, $\tilde{p}_0 = 0$, $\zeta_c = 0$, $\tilde{\zeta}_c = 0$, and $\beta_0 = 0$.
    Choose {\tt tol} and {\tt max\_itn}. \\
    5. {\bf for} $i = 1 \ldots$ {\tt max\_itn} {\bf do} \\
    \begin{tabular}[b]{ll}
    $\diamond$ \;  \;  $p_i = r_{i-1} + \beta_{i-1} p_{i-1}$; &
                       $\tilde{p}_i = \tilde{r}_{i-1} + \bar{\beta}_{i-1} \tilde{p}_{i-1}$ \\
    $\diamond$ \;  \;  $z_i = Ap_i$; &
                       $\tilde{z}_i = A^*\tilde{p}_i$ \\
    $\diamond$ \;  \;  $\zeta_i = \hat{C}^*z_i$; &
                       $\tilde{\zeta}_i = \check{C}^*\tilde{z}_i$ \\
    $\diamond$ \;  \;  $q_i = z_i - C\zeta_i$; &
                       $\tilde{q}_i = \tilde{z}_i - \tilde{C}\tilde{\zeta}_i$ \\
    $\diamond$ \;  \;  $\alpha_i = (\tilde{r}_{i-1}, r_{i-1}) / (\tilde{p}_i, q_i)$; &
                       $\tilde{\alpha}_i = \bar{\alpha}_i$ \\
    $\diamond$ \;  \;  $\zeta_c = \zeta_c + \alpha_i \zeta_i$; &
                       $\tilde{\zeta}_c = \tilde{\zeta}_c + \tilde{\alpha}_i \tilde{\zeta}_i$ \\
    $\diamond$ \;  \;  $x_i = x_{i-1} + \alpha_i p_i$; &
                       $\tilde{x}_i = \tilde{x}_{i-1} + \tilde{\alpha}_i \tilde{p}_i$ \\
    $\diamond$ \;  \;  $r_i = r_{i-1} - \alpha_i q_i$; &
                       $\tilde{r}_i = \tilde{r}_{i-1} - \tilde{\alpha}_i \tilde{q}_i$ \\
    \multicolumn{2}{l}{$\diamond$ \;  \;  {\bf if} $||r_i|| \leq $ {\tt tol} and
       $||\tilde{r}_i|| \leq $ {\tt tol} {\bf then break}} \\
    $\diamond$ \;  \;  $\beta_i = (\tilde{r}_i, r_i) / (\tilde{r}_{i-1}, r_{i-1})$ & \\
    \end{tabular} \\
    6. {\bf end for} \\
    7. $x_i = x_i - U \zeta_c$; \quad
       $\tilde{x}_i = \tilde{x}_i - \tilde{U}\tilde{\zeta}_c$
\end{figure}

%%%%%%%%%%%%%%%%%%%%%%%%%%%%%%%%%%%%%%%%%%%%%%%

\section{Recycling BiCGSTAB}\label{sec:rbicgstab}
In RBiCG~\cite{ahuja11phd, ahuja12rbicg}, the iteration vectors $p$, $\tilde{p}$, $r$, and $\tilde{r}$ are updated using the following recurrences:
\begin{align*}
	\begin{array}
	[l]{llllll}
    p_i & = & r_{i-1} + \beta_{i-1} p_{i-1}, & \tilde{p}_i & = & \tilde{r}_{i-1} + \tilde{\beta}_{i-1} \tilde{p}_{i-1}, \\
    r_i & = & r_{i-1} - \alpha_i B p_i, & \tilde{r}_i & = & \tilde{r}_{i-1} - \tilde{\alpha}_i \tilde{B}\tilde{p}_i,
    \end{array}
\end{align*}
%\begin{equation}\label{pequation}
%    p_i = r_{i-1} + \beta_{i-1} p_{i-1}, \quad \quad \tilde{p}_i = \tilde{r}_{i-1} + \tilde{\beta}_{i-1} \tilde{p}_{i-1},
%\end{equation}
%\begin{equation}\label{requationImpl}
%    r_i = r_{i-1} - \alpha_i B p_i, \quad \quad \tilde{r}_i = \tilde{r}_{i-1} - \tilde{\alpha}_i \tilde{B}\tilde{p}_i,
%\end{equation}
where $B = (I - C\hat{C}^*)A$ and $\tilde{B} = (I - \tilde{C}\check{C}^*)A^*$. We first give the polynomial representations of these iteration vectors.
\begin{theorem}\label{thm:polyExp}
    Let $r_i$, $p_i$, $\tilde{r}_i$, and $\tilde{p}_i$ be defined as above. Then, for the primary system
    \begin{equation*}\label{eq:polynomialRep}
        r_i = \Theta_i(B)r_0, \quad p_i = \Pi_{i-1}(B)r_0,
    \end{equation*}
    where $\Theta_i(K)$ and $\Pi_{i-1}(K)$ are $i$-th and $(i-1)$-st degree polynomials, for an arbitrary square matrix $K$, that satisfy the following recurrences: 
    \begin{align*}
        \Theta_i(K) = \Theta_{i-1}(K) - \alpha_i K \Pi_{i-1}(K), \\
        \Pi_{i-1}(K) = \Theta_{i-1}(K) + \beta_{i-1} \Pi_{i-2}(K).
    \end{align*}
    Similarly, for the dual system
\begin{equation*}
 	\tilde{r}_i = \bar{\Theta}_i(\tilde{B})\tilde{r}_0, \quad \tilde{p}_i = \bar{\Pi}_{i-1}(\tilde{B})\tilde{r}_0,
\end{equation*}
where $\bar{\Theta}_i(K)$ and $\bar{\Pi}_i(K)$ satisfy the following recurrences:
    \begin{align*}
        \bar{\Theta}_i(K) = \bar{\Theta}_{i-1}(K) - \bar{\alpha}_i K\bar{\Pi}_{i-1}(K), \\
        \bar{\Pi}_{i-1}(K) = \bar{\Theta}_{i-1}(K) + \bar{\beta}_{i-1} \bar{\Pi}_{i-2}(K).
    \end{align*}
\end{theorem}
\begin{proof}
	This can be proved by induction, following the derivation in~\cite{sonneveld89cgs} (Section 2; pages 37--40), but use $B$ instead of $A$ and $\tilde{B}$ instead of $A^*$.
\end{proof}

From RBiCG we know $r_i \perp \tilde{r}_j$ for $j<i$. Using Theorem \ref{thm:polyExp} we get that
\begin{equation*}
	(\bar{\Theta}_j(\tilde{B})\tilde{r}_0, \Theta_i(B)r_0) = 0 \text{ for } j < i.
\end{equation*}
This implies $\Theta_{i}(B)r_0$ $\perp$ $\mathcal{K}^i(\tilde{B}, \tilde{r}_{0})$, where $\tilde{r}_0$, $\tilde{B}\tilde{r}_0$, $\ldots$, $\tilde{B}^{i-1}\tilde{r}_0$ span the subspace $\mathcal{K}^i(\tilde{B}, \tilde{r}_{0})$. As observed in~\cite{vorst92bicgstab}, the above orthogonality conditions must be satisfied by other bases of $\mathcal{K}^i(\tilde{B}, \tilde{r}_{0})$, too. So, other polynomials can be used as well~\cite{zhang97gbicgstab}. That is,
\begin{align}\label{newPolyEquation}
	(\bar{\Omega}_j(\tilde{B})\tilde{r}_0, \Theta_i(B)r_0) = 0 \text{ for } j < i.
\end{align}
Similar to the derivation in~\cite{vorst92bicgstab}, we define
\begin{align*}
    	\begin{array}
        [l]{lll}
	\bar{\Omega}_i(\tilde{B}) & = & (I-\bar{\omega}_1 \tilde{B})(I-\bar{\omega}_2 \tilde{B}) \cdots (I-\bar{\omega}_i \tilde{B}),
    	\end{array}
\end{align*}
where $\omega_i$ is selected to minimize the residual $r_i$ w.r.t. $\omega_i$. Then, as first proposed in~\cite{sonneveld89cgs}, instead of (\ref{newPolyEquation}), we use the following form of inner product:
\begin{equation*}
	(\tilde{r}_0, {\Omega}_j(B)\Theta_i(B)r_0) = 0 \text{ for }  j < i,
\end{equation*}
with
\begin{align*}
    	\begin{array}
        [l]{lll}
	{\Omega}_i({B}) & = & (I-{\omega}_1 {B})(I-{\omega}_2 {B}) \cdots (I-{\omega}_i {B}).
    	\end{array}
\end{align*}
This inner product does not require the transpose of $B$, and hence, is appropriate when there is no dual system to solve. Computing the inner product in this fashion, we obtain the recycling BiCGSTAB algorithm (similar to the way BiCGSTAB is obtained from BiCG in~\cite{vorst92bicgstab}). We term our recycling BiCGSTAB as RBiCGSTAB. The algorithm is given in Algorithm 2. Some algorithmic improvements to make the code faster (similar to those discussed in section 6.2 of~\cite{mello2010opt}) are not given here.

Breakdowns in RBiCG (as discussed in the end of Section 3), lead to breakdowns in the RBiCGSTAB algorithm as well. This is similar to how breakdowns in BiCG lead to breakdowns in BiCGSTAB. A breakdown free BiCGSTAB algorithm is proposed in~\cite{cao1993breakdownbicgstab}, which uses the theory of formal orthogonal polynomials. The same theory can be applied to the RBiCGSTAB algorithm. 

The BiCGSTAB algorithm also breaks down when the minimization with respect to $\omega_i$ fails. This problem can be avoided by minimizing in two or more dimensions. This led to the development of BiCGSTAB2~\cite{gutnket93bicgstab2} and BiCGSTAB({\em l})~\cite{sleijpen93bicgstabl}. In~\cite{ahuja11phd}, a Recycling BiCGSTAB2 that can be extended to a Recycling BiCGSTAB({\em l}) is proposed.

\begin{figure}
{\bf Algorithm 2.} {\it RBiCGSTAB} \\
    1. Given $U$ (also $C = AU$) and $\tilde{U}$ (also $\tilde{C} = A^* \tilde{U}$) s.t. $C \perp_b \tilde{C}$, compute $D_c = \tilde{C}^*C$ and $\hat{C} = \tilde{C}\mathcal{D}_c^{-1}$. \\
    2. Choose $x_{-1}$; initialize $\tilde{r}_{-1}$ to a random vector; and compute $x_0$, $r_0$, $\tilde{r}_{0}$ using (\ref{minusXR}). \\
    3. {\bf if} $(r_0, \tilde{r}_0) = 0$ {\bf then} rechoose $x_{-1}$ or reinitialize $\tilde{r}_{-1}$ to avoid this condition. \\
    4. Set scalars $\beta_0$ and $\omega_0$ as well as vectors $p_0$, $q_0$, and $x_c$ to zero. \\
    5. Choose {\tt tol} and {\tt max\_itn}. \\
    6. {\bf for} $i = 1 \ldots$ {\tt max\_itn} {\bf do} \\
    \begin{tabular}[b]{l}
    	   $\diamond$\; \; $ p_i = r_{i-1} + \beta_{i-1}p_{i-1} - \beta_{i-1}\omega_{i-1}q_{i-1} $ \\
    	   $\diamond$\; \; $ q_i = Ap_i$ \\
    	   $\diamond$\; \; $ \zeta_i = \hat{C}^*q_i$ \\
    	   $\diamond$\; \; $ q_i = q_i - C \zeta_i$ \\
	   %%%%   
	   $\diamond$\; \; $ \alpha_{i} = \frac{\left(\tilde{r}_0,\ r_{i-1}\right)}{\left(\tilde{r}_0,\ q_i\right)} $ \\
	   $\diamond$\; \; $ s_i = r_{i-1} - \alpha_i q_i $ \\
	   $\diamond$\; \; $ t_i = A s_i $ \\
	   $\diamond$\; \; $ \gamma_i  = \hat{C}^*t_i $ \\
	   $\diamond$\; \; $ t_i = t_i - C \gamma_i$ \\
	   %%%%
	   $\diamond$\; \; $ \omega_i = \frac{(s_i,t_i)}{(t_i,t_i)} $ \\
	   $\diamond$\; \; $ x_i = x_{i-1} + \alpha_i p_i + \omega_i s_i $ \\
	   $\diamond$\; \; $ x_c = x_c + \alpha_i \zeta_i + \omega_i \gamma_i $ \\
	   $\diamond$\; \; $ r_i = r_{i-1} - \alpha_i q_i - \omega_i t_i $ \\
	   $\diamond$ \;  \;  {\bf if} $||r_i|| \leq $ {\tt tol} {\bf then break} \\
	   $\diamond$\; \; $ \beta_{i} = \frac{(\tilde{r}_0,\ r_{i})}{(\tilde{r}_0,\ r_{i-1})} \cdot \frac{\alpha_{i}}{\omega_{i}} $ \\
    \end{tabular} \\
    7. {\bf end for} \\ 
    8.  $x_i = x_i - U x_c$ \\
    \small{Note: When A does not change for multiple systems, several changes should be made to make this algorithm substantially more efficient. This is done in our implementation for the model reduction problem of Section \ref{sec:experiments}.}
\end{figure}

%%%%%%%%%%%%%%%%%%%%%%%%%%%%%%%%%%%%%%%%%%%%%%%%

\section{Numerical Experiments}\label{sec:analysis}
% It has been shown that deflation of the primary system residual is done by the left eigenvectors~\cite{sturler:householder}.
For BiCG, it has been shown that including a left eigenvector into the search space leads to the removal of the corresponding right eigenvector from the right residual (and vice versa)~\cite{desturler99householder}. In our experiments we demonstrate that recycling a left invariant subspace may improve the convergence rate in the RBiCGSTAB algorithm. We consider two examples. For the first example, we perform vertex centered finite volume discretization of the PDE
\begin{align*}
	\mathsf{-(u_x)_x - (u_y)_y + 10u_x -10u_y = 0},
\end{align*}
on a $42 \times 42$ grid unit square resulting in a $1600 \times 1600$ linear system. We use the following boundary conditions: $\mathsf{u}_\text{south} = 1, \mathsf{u}_\text{west} = 1, \mathsf{u}_\text{north} = 0, \mathsf{u}_\text{east} = 0$. We do not use a preconditioner in this example, the initial guess is a vector of all ones, and the relative convergence tolerance is $10^{-10}$.

For the second example, we perform finite difference discretization of the partial differential equation~\cite{vorst92bicgstab}
\begin{align*}
    \mathsf{-({A}v_x)_x - ({A}v_y)_y + {B}(x,y)v_x = {F}},
%    -(\mathcal{A}\vartheta_x)_x - (\mathcal{A}\vartheta_y)_y + \mathcal{B}(x,y)\vartheta_x = \mathcal{F},
\end{align*}
with $\mathsf{A}$ as shown in Figure \ref{figure:secondpde}, $\mathsf{B(x,y) = 2e^{2(x^2 + y^2)}}$, and $\mathsf{F} = 0$ everywhere except in a small square in the center where $\mathsf{F} = 100$ (see Figure \ref{figure:secondpde}). The $(0,1) \times (0,1)$ domain is discretized on a $129 \times 129$ grid resulting in a  $16129 \times 16129$ linear system. We use the following boundary conditions:
\begin{align*}
    \begin{array}
    [l]{l}
    \mathsf{v(0, y) = v(1, y) = v(x, 0) = 1}, \\
    \mathsf{v(x, 1) = 0}.
    \end{array}
\end{align*}
% We set the mesh width to $h = 1/128$,
We use an ILUTP~\cite{saad03book} preconditioner with a drop tolerance of $0.1$ (split-preconditioned). The initial guess is $0.5$ times a vector of all ones, and the relative convergence tolerance is $10^{-8}$.

\begin{figure}
    \centering
    \includegraphics[angle=90,scale=0.5]{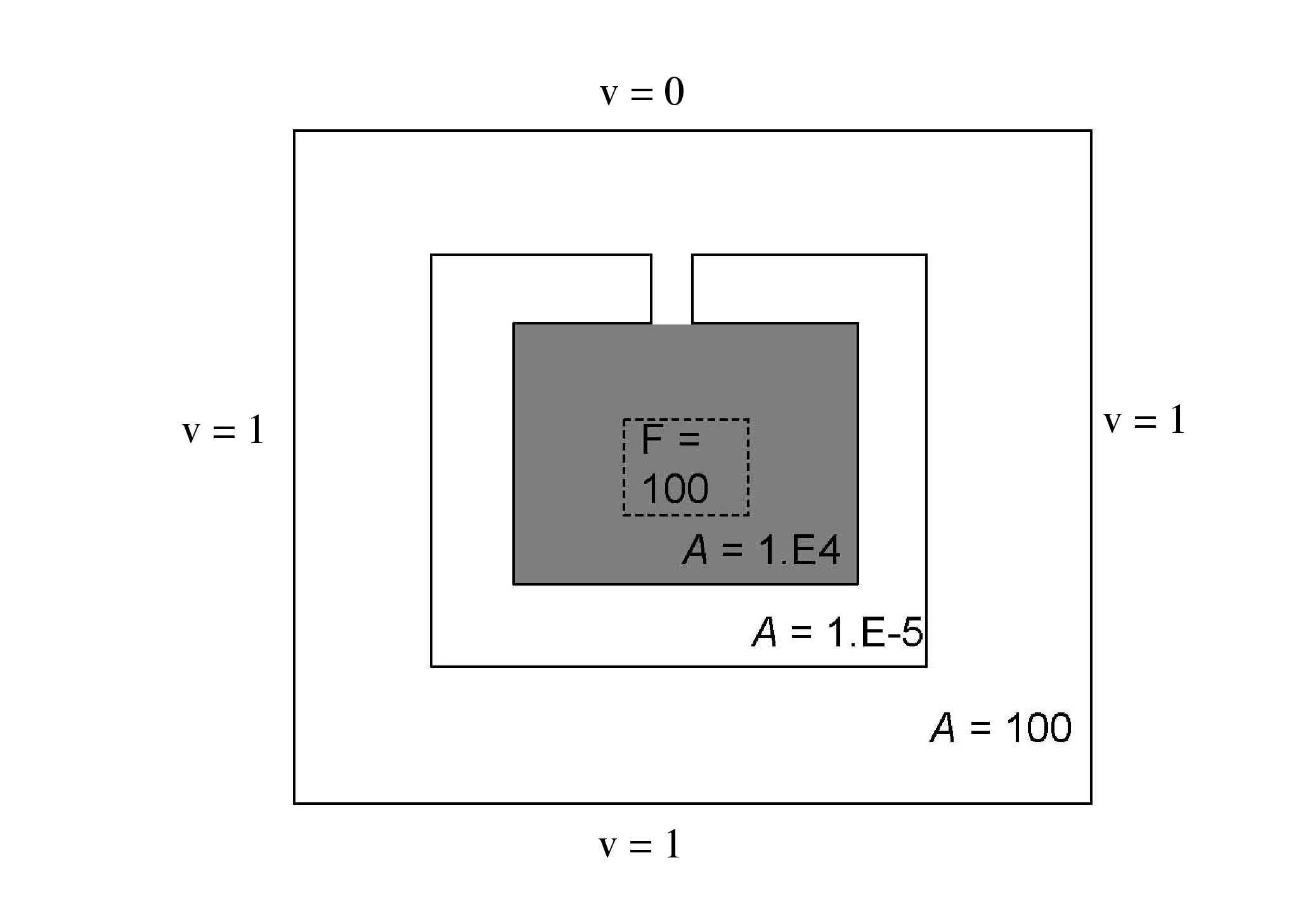}
    \caption{Coefficients for the PDE.}\label{figure:secondpde}
\end{figure}

\begin{figure}
    \subfigure[Example 1: Left eigenvectors not needed, but recycling effective.]{
        	\scalebox{0.6}{\includegraphics{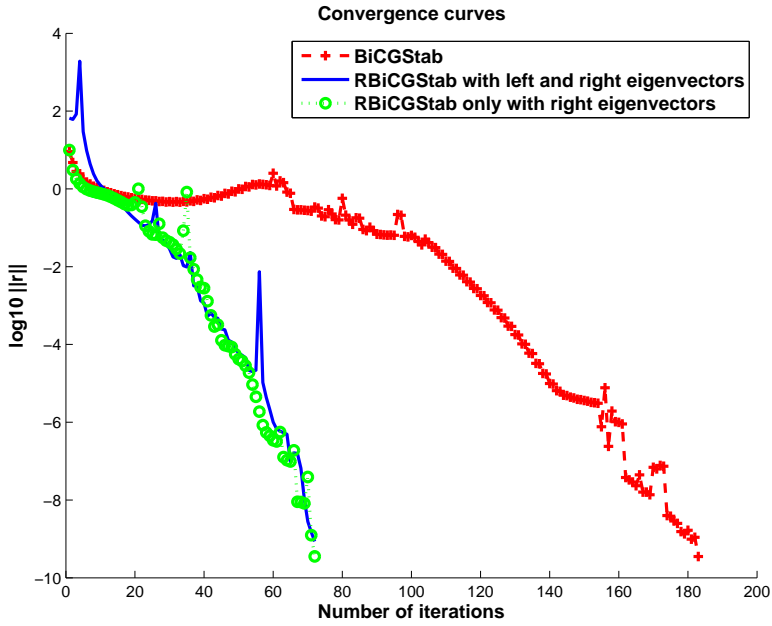}}
    }
    \subfigure[Example 2: Left eigenvectors needed for recycling to be effective.]{
        	\scalebox{0.6}{\includegraphics{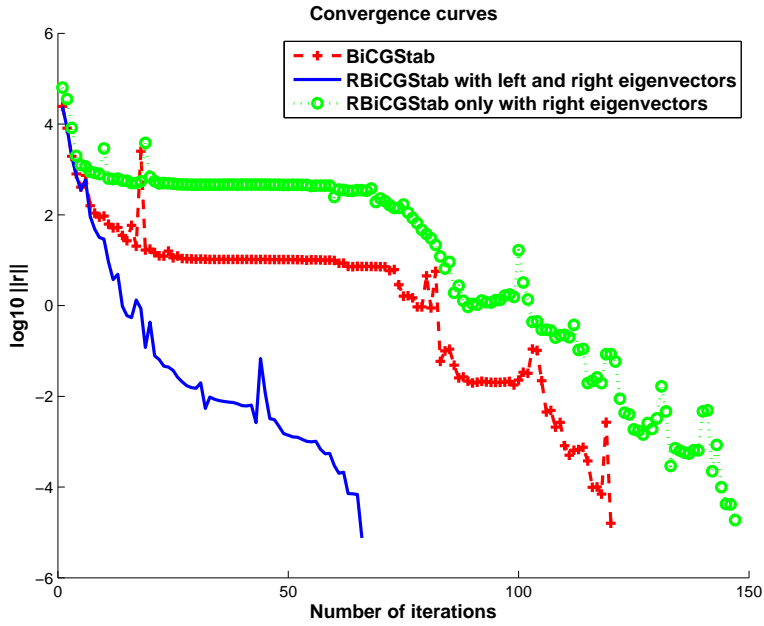}}
    }
    \caption{Convergence curves for two examples using RBiCGSTAB. The 2nd example demonstrates that recycling a left invariant subspace may improve the convergence rate in the RBiCGSTAB algorithm.}\label{figure:needLeftEigVecs}
\end{figure}

For each example we do three experiments. First, we solve the system without recycling. Second, we use the right invariant subspace (corresponding to the smallest magnitude eigenvalues) as the recycle space. This is implemented by setting $\tilde{U} = U$. Finally, we use both the left and right invariant subspaces (again, corresponding to the smallest magnitude eigenvalues) as the recycle space.

For the first example's second set of experiments, we use five exact right eigenvectors computed using the MATLAB function {\tt eigs}. For the first example's third set of experiments, we use five exact left eigenvectors and five exact right eigenvectors (for a total of ten), again computed using the MATLAB function {\tt eigs}.

For the second example's second set of experiments, we use twenty approximate right eigenvectors obtained by solving the problem twice with RBiCG. For the second example's third set of experiments, we use twenty approximate left eigenvectors and twenty approximate right eigenvectors (for a total of forty), again obtained by solving the problem twice with RBiCG.

The results are shown in Figures \ref{figure:needLeftEigVecs} (a) and (b). For the first example, using the right invariant subspace or using both the left and right invariant subspaces works equally well. However, for the second example, we see that using only the right invariant subspace leads to convergence that is worse than BiCGSTAB without recycling, and much worse than RBiCGSTAB using both the left and right invariant subspaces. This shows that recycling a left invariant subspace may improve the convergence rate in the RBiCGSTAB algorithm. 

Next, we analyze why the first example does not need left invariant subspace and the second example does, by considering the cosines of the principal angles between the left and right invariant subspaces associated with the ten smallest magnitude eigenvalues. Table \ref{table:analysis} lists these angles.

\begin{table}[tbh]
   \footnotesize
    \begin{center}
        \setlength{\extrarowheight}{1.5pt}
        \begin{tabular}{|m{1in}|m{1in}|}
            \hline
            {Example 1} & {Example 2} \\         	   
            \hline
	   0.9998	&	0.1039    \\
	   0.9837	&	0.0302    \\
	   0.9207	&	0.0195    \\
	   0.7617	&	0.0106	  \\
	   0.4352	&	0.0089	  \\
	   0.3987	&	0.0049	  \\
	   0.2273	&	0.0043	  \\
	   0.0963	&	0.0027	  \\
	   0.0235	&	0.0018	  \\
	   0.0064	&	0.0012	  \\
          \hline 
        \end{tabular}
    \end{center}
         \quad \\
    \caption{\label{table:analysis} For each example, we give the cosines of the principal angles between the  exact left and right invariant subspaces of dimension $10$, associated with the smallest magnitude eigenvalues. For the first example, we compute invariant subspaces of the matrix obtained after discretization. While for the second example, we compute invariant subspaces of the preconditioned matrix since we split-precondition the linear system obtained after discretization.}
\end{table}

From the table, we see that the principal angles between the left and right invariant subspaces for the second example are substantially larger than those for the first example (since the cosines of the principal angles is lesser for second example as compared with the first example). Since for a normal matrix left and right invariant subspaces are identical, we conclude that the second example is ``more'' non-normal than the first.    

%%%%%%%%%%%%%%%%%%%%%%%%%%%%%%%%%%%%%%%%%%%%%%%%
\section{Applications}\label{sec:appl}
We first discuss several techniques for parametric model order reduction and the one we are using (in Section \ref{sec:pmor}) followed by the description of how recycling BiCGSTAB is applied for sequences of linear systems arising in parametric model order reduction (in Section \ref{sec:experiments}). 

\subsection{Parametric Model Order Reduction}\label{sec:pmor}
Numerical simulation is an essential tool for solving science and engineering problems. However, simulating large-scale models leads to overwhelming demands on computational resources. This is the main motivation for model reduction. The goal is to produce a surrogate model of much smaller dimension that provides a high-fidelity approximation of the input-output behavior of the original model. Often the models have design parameters associated with them, e.g., boundary conditions, geometry, material properties etc. Changes in these design parameters require generation of new reduced models, which makes the model reduction process very cumbersome. One practical application where such a challenge arises is micro-electro-mechanical systems (MEMS) design~\cite{feng2012Mems, feng2013book}. The goal of parametric model order reduction (PMOR)~\cite{baur2011PMOR,
% baur2011mems,
feng:proceedings} is to generate a reduced model such that parametric dependence, as in the original model, is preserved (or retained).

We focus on physical processes that are modeled as parameterized partial differential equations (PDEs). For PMOR, the PDE is first semi-discretized using classical techniques (e.g., finite differences, finite elements, etc.), and then model reduction is applied to the resulting parameterized state-space model~\cite{baur2011PMOR}:
\begin{align}\label{eq:dynamicalSys1}
    \begin{array}
        [l]{ll}%
        \mathsf{G}:
        \begin{cases}
           \mathsf{E(p)\ {\dot x}(t) = A(p)\ x(t) + B(p)\ u(t)} \\
            \mathsf{y(t) = C(p)\ x(t)}
        \end{cases}
    \end{array}
\end{align}
or
\begin{align}\label{eq:dynamicalSys2}
    \mathsf{G(s, p) = C(p)\left(sE\left(p\right) - A\left(p\right)\right)^{-1}B(p)}.
\end{align}
Sometimes, the physical process is directly available in the form of a parametrized state-space model. Equations (\ref{eq:dynamicalSys1}) -- (\ref{eq:dynamicalSys2}) represent a multiple input multiple output (MIMO) linear dynamical system, where $\mathsf{p}$ is the parameter vector; $\mathsf{u(t) = [u_1(t), \ldots}$, $\mathsf{u_m(t)]^T}$: $\mathsf{\mathbb{R} \rightarrow \mathbb{R}^m}$ is the input; $\mathsf{y(t)}$: $\mathsf{\mathbb{R} \rightarrow \mathbb{R}^l}$ is the output; $\mathsf{x(t)}$: $\mathsf{\mathbb{R} \rightarrow \mathbb{R}^n}$ is the state vector; $\mathsf{E(p)}$, $\mathsf{A(p) \in \mathbb{R}^{n \times n}}$, $\mathsf{B(p) \in \mathbb{R}^{n \times m}}$, and $\mathsf{C(p) \in \mathbb{R}^{l \times n}}$ are the system matrices; and $\mathsf{s}$ is the frequency domain variable corresponding to $\mathsf{t}$ in the time domain.

Above, (\ref{eq:dynamicalSys1}) denotes the dynamical system, and (\ref{eq:dynamicalSys2}) gives the transfer function of the system obtained after Laplace transformation. By a common abuse of notation, we denote both with $\mathsf{G}$. The dimension of the underlying state-space, $\mathsf{n}$, is called the dimension or order of $\mathsf{G}$. Unless explicitly stated, for the rest of this paper all dynamical systems are assumed to be of the above type.

There are various ways of performing PMOR~\cite{benner2013survey, baur2011PMOR, patera06reducedbasis, benner2014moment}. This includes moment matching, local $\Htwo$-optimality, and reduced basis approaches. For this work, we focus on moment matching based PMOR because of its flexibility (few limits on the system properties) and low computational cost in many industrial applications.

%%%%%%%%%%%%%%%%%%%%% KAPIL May 23rd 2014 NEED TO FILL THIS WHEN REVIWERS COMMENTS COME%%%%%%%%%%%%%%%%
%
% ??? Moment matching based PMOR algorithms ???
%
%%%%%%%%%%%%%%%%%%%%% KAPIL May 23rd 2014 NEED TO FILL THIS WHEN REVIWERS COMMENTS COME%%%%%%%%%%%%%%%%

Moment matching based PMOR algorithms~\cite{benner2014moment, feng2012Mems} require solution of sequences of linear systems of the type (\ref{eq:initial3}), which is a key bottleneck when using these algorithms for reducing larger models. Specifically, the systems have the form as follows: 
\begin{align*}
    {A}^{(1)}\ x^{(1, 1)} & = {b}^{(1, 1)} \\
    {A}^{(1)}\ x^{(1, 2)} & = {b}^{(1, 2)} \\
    {A}^{(1)}\ x^{(1, 3)} & = {b}^{(1, 3)} \\
    & \vdotswithin{=}   \\
    {A}^{(2)}\ x^{(2, 1)} & = {b}^{(2, 1)} \\
    {A}^{(2)}\ x^{(2, 2)} & = {b}^{(2, 2)} \\
    {A}^{(2)}\ x^{(2, 3)} & = {b}^{(2, 3)} \\
    & \vdotswithin{=}  
\end{align*} 
In the example used here, each matrix has the form
\begin{align*}
{A}^{(i)} = A_0 + \sigma_iA_1 + \delta_i A_2 \quad \text{with} \quad i \in \mathbb{N}.
\end{align*}
Please note that $\sigma_iA_1$ and $\delta_i A_2$ are small perturbations to $A_0$. As discussed in the introduction, our goal here is to use Krylov subspace recycling (recycling BiCGSTAB specifically) to efficiently solve such systems.

%% NOT NEEDED HERE: ??? Why not GCRO-DR and its variant to to be emphasised again ???

%%%%%%%%%%%%%%%%%%%%%%%%%%%%%%%%%%%%%%%%%%%%%%%%

\subsection{Application to PMOR}\label{sec:experiments}
Our test dynamical system comes from a silicon nitride membrane model~\cite{bechtold2010SiN}. Such a membrane can be part of many devices, e.g., a gas sensor chip, a microthruster, an optical filter etc.
%Finite element discretization of the heat transfer equation that governs the membrane, leads to the following parameterized system:
%\begin{align*}\label{eq:silconeNitride}
%    \begin{array}
%        [r]{rrl}%
%           \left(E_0 + \rho c_pE_1\right){\dot x} + (K_0 + \kappa K_1 + h K_2) x & = & Bu(t)\\
%            y & = & Cx.
%    \end{array}
%\end{align*}
%This equation consists of four parameters: $\rho$ the mass density, $c_p$ the specific heat capacity, $\kappa$ the thermal conductivity, and $h$ the heat transfer coefficient. As earlier, $u$ and $y$ are input and output respectively. The state variable $x$ defines the temperature distribution.
We use the moment matching based PMOR algorithm, described in~\cite{benner2014moment, feng2012Mems}, to compute a reduced model. This leads to a sequence of linear systems of the form (\ref{eq:initial3}) and size $60,020$.

Whenever the matrix changes in the sequence, we call RBiCG to perform the linear solve. This helps to approximate both left and right invariant subspaces, which are not easily available from the RBiCGSTAB iterations (sometimes a left invariant subspace is available from a right invariant subspace~\cite{abdel2007lefteig, morgan2010nlandr}). The primary system right-hand side comes from the PDE. We take a vector of all ones as the dual system right-hand side. We call RBiCGSTAB for all remaining systems with the same matrix. This corresponds to linear systems where only the right-hand sides change. This is an effective strategy because it has been shown that the recycle space can be useful for multiple consecutive systems~\cite{parks06gcrodr, kilmer2006dot, morgan2010nlandr,  abdel2014rhs}. Moreover, using RBiCG for all systems will be expensive since an unnecessary dual system would be solved at each step in the sequence. It needs to be emphasized here that these is a need to implement this efficiently (as discussed in Algorithm 2). 

While solving a linear system with RBiCG, Lanczos vectors are generated at each iterative step. These Lanczos vectors are used to build the recycle space. We have the flexibility in deciding when to build the recycle space. One option is to wait for the RBiCG to converge, save all the Lanczos vectors, and then build the recycle space. The problem with this approach is that this requires large amounts of memory and may be computationally expensive as well. Instead, we divide the RBiCG iteration in cycles of a certain number of iterations (to be chosen, e.g., $50$ iterations). At the end of each cycle, we use the stored Lanczos vectors from that cycle to build or improve the recycle space, and then discard these Lanczos vectors, except for the last few that are needed to continue the Lanczos iteration. At the end of the first cycle, we may build a new recycle space or update a recycle space constructed for a previous linear system. In the following discussion, we denote the length of a cycle by $s$.

For this experiment, we take $s = 25$ and $k = 20$ (the number of vectors selected for recycling as defined earlier in Section 3). These values are chosen based on experience with other recycling algorithms~\cite{parks06gcrodr}. The linear systems are split-preconditioned with an incomplete LU preconditioner with threshold and pivoting (ILUTP)~\cite{saad03book}. The drop tolerance is taken as $10^{-4}$. For RBiCG, we take a vector of all zeros as the initial guess for both the primary system and the dual system. For RBiCGSTAB we take a vector of all zeros as the initial guess as well (we are only solving the primary system in RBiCGSTAB). Using the solution from the previous linear system an initial guess leads to poor convergence for BiCGSTAB with and without recycling. In general, a better initial guess may be based on knowledge of the system. The relative convergence tolerance for the iterative solves is taken as $10^{-8}$.

%%%%%%%%%%%%%%%%%%%%% KAPIL May 23rd 2014 NEED TO FILL THIS WHEN REVIWERS COMMENTS COME%%%%%%%%%%%%%%%%
%
% ??? Details of parameters used in model reduction algorithm like V=forthognalize(VV,0.0000001) ???
%
%%%%%%%%%%%%%%%%%%%%% KAPIL May 23rd 2014 NEED TO FILL THIS WHEN REVIWERS COMMENTS COME%%%%%%%%%%%%%%%%

The number of matrix-vector products required to solve systems $1$ through $63$ are given in Figure \ref{figure:matvecs}, and the corresponding timing data is given in Figure \ref{figure:time}. In both the figures, the peaks in the recycling BiCGSTAB plot correspond to when the matrix changes and RBiCG is called (three times; at the 1st, 22nd, and 43rd linear system). For all other steps, when only the right-hand side changes, RBiCGSTAB is called. When recomputing the recycle space for the 22nd system we use the recycle space generated while solving the 1st system, and make it better. This is evident in Figure 3 and 4 where both the number of matrix-vector products and time for solving the 22nd linear system are less as compared with the 1st system. The same process happens when recomputing the recycle space for the 43rd system. Here, the recycle space from the 22nd linear system is improved. Again, the Figures 3 and 4 show the decrease in the matrix-vector product count and time for the 43rd system as compared with the 22nd system.

First, we compare our results with BiCGSTAB. The RBiCG and RBiCGSTAB combination requires about 40\% fewer matrix-vector products in total. Also, computing the reduced model with recycling takes about 35\% less time than without recycling. This demonstrates the effectiveness of recycling Krylov subspaces for PMOR.

Second, we compare our results with GCRO-DR. Although the RBiCG and RBiCGSTAB combination does not beat GCRO-DR in the number of matrix vector products, it is 10\% more efficient than GCRO-DR in time and reduces storage requirements. This is similarly seen in BiCGSTAB and GMRES comparison for some examples. That is, although BiCGSTAB is expensive than GMRES in the number of matrix vector products, it is cheaper than GMRES in time. The reason being that BiCGSTAB is cheaper than GMRES with respect to the orthogonalizations.

\begin{figure}
    \centering
    \includegraphics[scale=0.7]{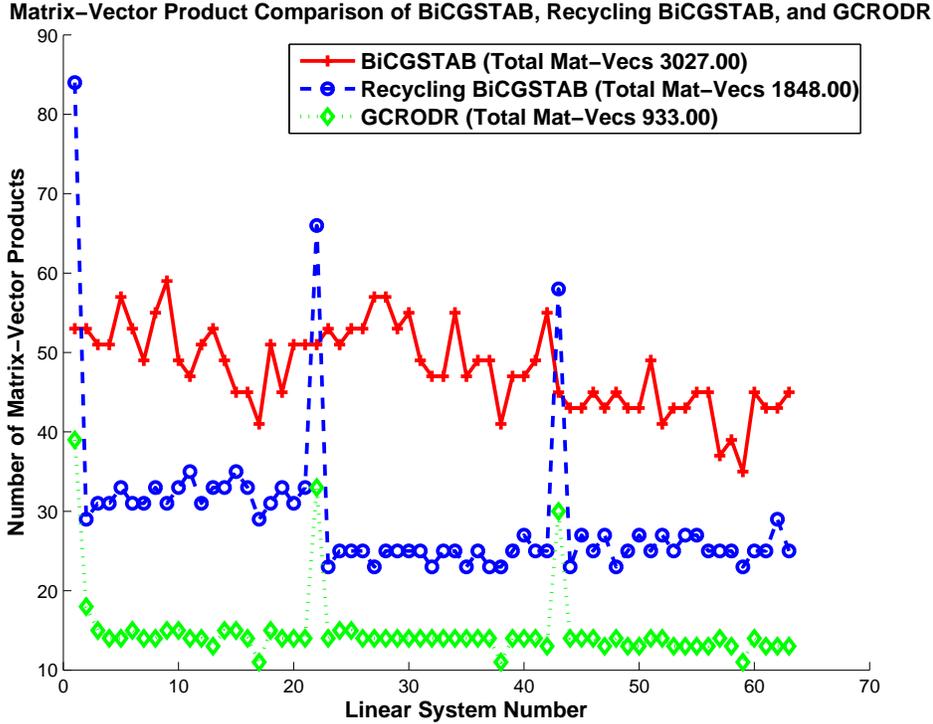}
    \caption{Comparison of matrix-vector product count when using BiCGSTAB, RBiCGSTAB, and GCDR-DR as the linear solvers for PMOR.}\label{figure:matvecs}
\end{figure}

\begin{figure}
    \centering
    \includegraphics[scale=0.7]{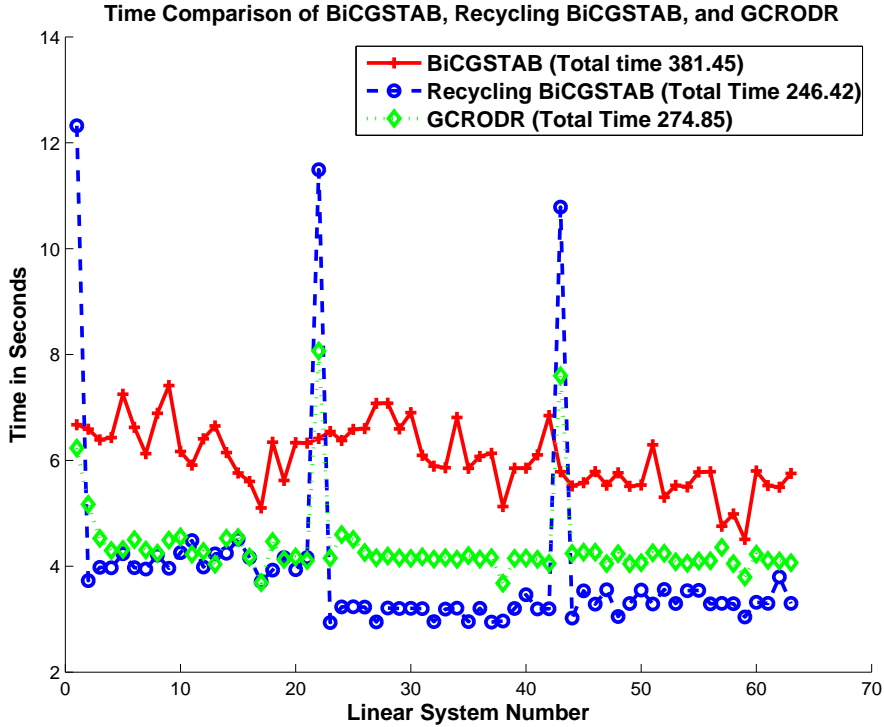}
    \caption{Comparison of time (in seconds) when using BiCGSTAB, RBiCGSTAB, and GCRO-DR as the linear solvers for PMOR.}\label{figure:time}
\end{figure}

%%%%%%%%%%%%%%%%%%%%%%%%%%%%%%%%%%%%%%%%%%%%%%%%

\section{Conclusions}\label{sec:conclusion}

For solving linear systems with non-symmetric matrices, B\-i\-C\-G\-S\-T\-A\-B is one of the best available algorithms. As compared with GMRES, which is the most commonly used algorithm for such linear systems, it has the advantage of a short term recurrence, and hence, does not suffer from storage issues.

For solving sequences of linear systems with non-symmetric matrices, it is advantageous to use Krylov subspace recycling for the BiCGSTAB algorithm, and hence we propose the RBiCGSTAB algorithm. We have demonstrated the usefulness of RBiCGSTAB for a parametric model order reduction example.

Here, we have used RBiCG to solve a linear system and generate the recycle space. We have then used RBiCGSTAB, that uses this generated space, to solve the subsequent systems until the matrix changes. RBiCG in this approach can be replaced with other solvers like GCRO-DR~\cite{parks06gcrodr} as well.

In the future, we plan to test RBiCGSTAB for other application areas (e.g., acoustics problems). We also plan to extend the recycling framework of RBiCGSTAB to BiCGSTAB({\em l})~\cite{sleijpen93bicgstabl} and IDR~\cite{wesseling80idr, sonneveld2008idr}.
%({\em l}).
In section \ref{sec:rbicgstab}, we saw that BiCGSTAB (and RBiCGSTAB) performs one-dimensional minimization of the residual. This minimization can be done in higher dimensions as well (say {\em l}), leading to BiCGSTAB({\em l})\footnote{One-dimensional minimization was first extended to two dimensions in~\cite{gutnket93bicgstab2}}. Like BiCGSTAB, the induced dimension reduction (IDR) method involves a short term recurrence and has been shown to perform better than BiCGSTAB in many cases~\cite{sonneveld2008idr}.

\vspace{6pt}
\noindent
{\bf Acknowledgments.} We thank the anonymous reviewers for their careful and helpful suggestions, which greatly helped us to improve this paper.

\bibliographystyle{abbrv}
\bibliography{rbicgstab}

\begin{thebibliography}{10}

\bibitem{abdel2014rhs}
A.~M. Abdel-Rehim, A.~Stathopoulos, and K.~Orginos.
\newblock Extending the {eigCG} algorithm to nonsymmetric {L}anczos for linear
  systems with multiple right-hand sides.
\newblock {\em Numerical Linear Algebra with Applications}, 21(4):473--493,
  2014.

\bibitem{abdel2007lefteig}
A.~M. Abdel-Rehim, W.~Wilcox, and R.~B. Morgan.
\newblock Deflated {BiCGStab} for linear equations in {QCD} problems.
\newblock In {\em Proceedings of Science, LAT2007}, pages 026/1--026/7, 2007.

\bibitem{ahuja09ms}
K.~Ahuja.
\newblock Recycling bi-{L}anczos algorithms: {BiCG}, {CGS}, and {BiCGSTAB}.
\newblock Master's thesis, Department of Mathematics, Virginia Tech, August
  2009.
\newblock Advised by E. de Sturler. Available from
  \url{http://scholar.lib.vt.edu/theses/available/etd-08252009-161256/}.

\bibitem{ahuja11phd}
K.~Ahuja.
\newblock {\em Recycling Krylov Subspaces and Preconditioners}.
\newblock PhD thesis, Department of Mathematics, Virginia Tech, October 2011.
\newblock Advised by E. de Sturler. Available from
  \url{http://scholar.lib.vt.edu/theses/available/etd-11112011-010340/}.

\bibitem{ahuja12rbicg}
K.~Ahuja, E.~de~Sturler, S.~Gugercin, and E.~Chang.
\newblock Recycling {BiCG} with an application to model reduction.
\newblock {\em SIAM Journal on Scientific Computing}, 34(4):A1925--A1949, 2012.

\bibitem{bank1993csb}
R.~E. Bank and T.~F. Chan.
\newblock An analysis of the composite step biconjugate gradient method.
\newblock {\em Numer. Math.}, 66:295--319, 1993.

\bibitem{baur2011PMOR}
U.~Baur, C.~Beattie, P.~Benner, and S.~Gugercin.
\newblock Interpolatory projection methods for parameterized model reduction.
\newblock {\em SIAM Journal on Scientific Computing}, 33(5):2489--2518, 2011.

\bibitem{bechtold2010SiN}
T.~Bechtold, D.~Hohlfeld, E.~Rudnyi, and M.~G\"{u}nther.
\newblock Efficient extraction of thin-film thermal parameters from numerical
  models via parametric model order reduction.
\newblock {\em Journal of Micromechanics and Microengineering}, 20(4):045030
  (13pp), 2010.

\bibitem{feng2013book}
T.~Bechtold, G.~Schrag, and L.~Feng, editors.
\newblock {\em System-Level Modeling of {MEMS}}.
\newblock Advanced Micro {\&} Nanosystems. Wiley-VCH, 2013.

\bibitem{benner2014moment}
P.~Benner and L.~Feng.
\newblock A robust algorithm for parametric model order reduction based on
  implicit moment matching.
\newblock In A.~Quarteroni and G.~Rozza, editors, {\em Reduced Order Methods
  for Modeling and Computational Reduction}, volume~9 of {\em MS\&A Series},
  pages 159--186. Springer, 2014.

\bibitem{benner2013survey}
P.~Benner, S.~Gugercin, and K.~Willcox.
\newblock A survey of model reduction methods for parametric systems.
\newblock Technical Report MPIMD/13-14, Max Planck Institute Magdeburg, August
  2013.

\bibitem{cao1993breakdownbicgstab}
Z.-H. Cao.
\newblock Avoiding breakdown in variants of the {BI-CGSTAB} algorithm.
\newblock {\em Linear Algebra and its Applications}, 263(0):113--132, 1997.

\bibitem{desturler99householder}
E.~de~Sturler.
\newblock {BiCG} explained.
\newblock In {\em Householder Symposium XIV, Proceedings of the Householder
  International Symposium in Numerical Algebra, Chateau Whistler, Whistler, BC,
  Canada}, June 13--19, 1999.

\bibitem{feng:proceedings}
L.~Feng, P.~Benner, and J.~Korvink.
\newblock Parametric model order reduction accelerated by subspace recycling.
\newblock In {\em Proceedings of 48th IEEE Conference on Decision \& Control
  and 28th Chinese Control Conference}, pages 4328--4333, 2009.

\bibitem{feng2012Mems}
L.~Feng, P.~Benner, and J.~Korvink.
\newblock Subspace recycling accelerates the parametric macro-modeling of
  {MEMS}.
\newblock {\em International Journal for Numerical Methods in Engineering},
  94(1):84--110, 2013.

\bibitem{freund:lookahead}
R.~W. Freund, M.~H. Gutknecht, and N.~M. Nachtigal.
\newblock An implementation of the look-ahead {L}anczos algorithm for
  non-{H}ermitian matrices.
\newblock {\em SIAM Journal on Scientific Computing}, 14(1):137--158, 1993.

\bibitem{greenbaum1997book}
A.~Greenbaum.
\newblock {\em Iterative Methods for Solving Linear Systems}.
\newblock SIAM, 1997.

\bibitem{gutnket93bicgstab2}
M.~H. Gutknecht.
\newblock Variants of {BICGSTAB} for matrices with complex spectrum.
\newblock {\em SIAM Journal on Scientific and Statistical Computing},
  14:1020--1033, 1993.

\bibitem{gutnket1997survey}
M.~H. Gutknecht.
\newblock Lanczos-type solvers for nonsymmetric linear systems of equations.
\newblock {\em Acta Numerica}, 6:271--397, 1997.

\bibitem{gutnket2014bicg}
M.~H. Gutknecht.
\newblock Deflated and augmented {K}rylov subspace methods: A framework for
  deflated {BiCG} and related solvers.
\newblock Available from \url{http://www.sam.math.ethz.ch/~mhg/}, 2014.

\bibitem{kilmer2006dot}
M.~E. Kilmer and E.~de~Sturler.
\newblock Recycling subspace information for diffuse optical tomography.
\newblock {\em SIAM Journal on Scientific Computing}, 27(6):2140--2166, 2006.

\bibitem{lanczos1952base}
C.~Lanczos.
\newblock Solution of systems of linear equations by minimized iterations.
\newblock {\em Journal of Research of the National Bureau of Standards},
  49:33--53, 1952.

\bibitem{mello2010opt}
L.~A.~M. Mello, E.~de~Sturler, G.~H. Paulino, and E.~C.~N. Silva.
\newblock Recycling {K}rylov subspaces for efficient large-scale electrical
  impedance tomography.
\newblock {\em Computer Methods in Applied Mechanics and Engineering},
  199(49):3101 -- 3110, 2010.

\bibitem{morgan2010nlandr}
R.~B. Morgan and D.~A. Nicely.
\newblock Restarting the nonsymmetric {L}anczos algorithm for eigenvalues and
  linear equations including multiple right-hand sides.
\newblock {\em SIAM Journal on Scientific Computing}, 33(5):3037--3056, 2011.

\bibitem{parks06gcrodr}
M.~L. Parks, E.~de~Sturler, G.~Mackey, D.~D. Johnson, and S.~Maiti.
\newblock Recycling {K}rylov subspaces for sequences of linear systems.
\newblock {\em SIAM Journal on Scientific Computing}, 28(5):1651--1674, 2006.

\bibitem{patera06reducedbasis}
A.~T. Patera and G.~Rozz.
\newblock Reduced basis approximation and a posteriori error estimation for
  parameterized partial differential equations.
\newblock Version 1.0, Copyright MIT 2006, to appear in (tentative rubric) MIT
  Pappalardo Graduate Monographs in Mechanical Engineering. Available at
  http://augustine.mit.edu, 2006.

\bibitem{saad03book}
Y.~Saad.
\newblock {\em Iterative Methods for Sparse Linear Systems}.
\newblock Society for Industrial and Applied Mathematics, 3600 Market Street,
  Philadelphia, PA 19104-2688, USA, 2nd edition, 2003.

\bibitem{saad86gmres}
Y.~Saad and M.~H. Schultz.
\newblock {GMRES}: A generalized minimal residual algorithm for solving
  nonsymmetric linear systems.
\newblock {\em SIAM Journal on Scientific and Statistical Computing},
  7(3):856--869, 1986.

\bibitem{sleijpen93bicgstabl}
G.~L.~G. Sleijpen and D.~R. Fokkema.
\newblock {BiCGstab}({\em l}) for linear equations involving unsymmetric
  matrices with complex spectrum.
\newblock {\em Electronic Transactions on Numerical Analysis}, 1:11--32, 1993.

\bibitem{sonneveld89cgs}
P.~Sonneveld.
\newblock {CGS}, a fast {L}anczos-type solver for nonsymmetric linear systems.
\newblock {\em SIAM Journal on Scientific and Statistical Computing},
  10(1):36--52, 1989.

\bibitem{sonneveld2008idr}
P.~Sonneveld and M.~B. van Gijzen.
\newblock {IDR}(s): A family of simple and fast algorithms for solving large
  nonsymmetric systems of linear equations.
\newblock {\em SIAM Journal on Scientific Computing}, 31(2):1035--1062, 2008.

\bibitem{vorst92bicgstab}
H.~A. van~der Vorst.
\newblock Bi-{CGSTAB}: a fast and smoothly converging variant of {B}i-{CG} for
  the solution of nonsymmetric linear systems.
\newblock {\em SIAM Journal on Scientific and Statistical Computing},
  13(2):631--644, 1992.

\bibitem{wesseling80idr}
P.~Wesseling and P.~Sonneveld.
\newblock Numerical experiments with a multiple grid and a preconditioned
  {L}anczos type method.
\newblock In R.~Rautmann, editor, {\em Approximation Methods for Navier-Stokes
  Problems}, volume 771 of {\em Lecture Notes in Mathematics}, pages 543--562.
  Springer Berlin Heidelberg, 1980.

\bibitem{zhang97gbicgstab}
S.-L. Zhang.
\newblock {GPBi-CG}: Generalized product-type methods based on {Bi-CG} for
  solving nonsymmetric linear systems.
\newblock {\em SIAM Journal on Scientific Computing}, 18(2):537--551, 1997.

\end{thebibliography}

\end{document}